\documentclass[12pt]{amsart}

\usepackage{amsmath,amssymb,amsthm}
\usepackage{graphicx}
\usepackage{enumerate}
\usepackage{color}
\usepackage[german, english]{babel}
\usepackage{hyperref}
\usepackage[all]{xy}

\newtheorem{thm}{Theorem}[]
\newtheorem*{thm*}{Theorem}

\newtheorem{cor}[thm]{Corollary}

\newtheorem{ex}[thm]{Example}
\newtheorem{rem}[thm]{Remark}
\newtheorem{defn}[thm]{Definition}

\newcommand{\param}{{\mathchoice{\mkern1mu\mbox{\raise2.2pt\hbox{$
\centerdot$}}
\mkern1mu}{\mkern1mu\mbox{\raise2.2pt\hbox{$\centerdot$}}\mkern1mu}{
\mkern1.5mu\centerdot\mkern1.5mu}{\mkern1.5mu\centerdot\mkern1.5mu}}}

\renewcommand{\setminus}{{\smallsetminus}}

\renewcommand \color [2][]{}

\begin{document}

\title       {Critical levels and Jacobi fields in a complex of cycles}
\author   {Ingrid Irmer}
\address {Department of Mathematics\\
             Florida State University\\
             208 Love Building\\
1017 Academic Way\\
Tallahassee, FL 32306-4510}
\email      {iirmer@math.fsu.edu}
\maketitle

\begin{abstract}
In this paper it is shown that the space of tight geodesic segments connecting any two vertices in a complex of cycles has finite, uniformly bounded dimension. The dimension is defined in terms of a discrete analogue of Jacobi fields, which are explicitly constructed and shown to give a complete description of the entire space of tight geodesics. Jacobi fields measure the extent to which geodesic stability breaks down. Unlike most finiteness properties of curve complexes, the arguments presented here do not rely on hyperbolicity, but rather on structures similar to Morse theory.
\end{abstract}

\section{Introduction}
Suppose $S$ is a closed, oriented, connected surface of genus at least two. The complex of cycles, $\mathcal{C}(S,\alpha)$ is a variant of Harvey's complex of curves, where vertices represent multicurves in the primitive homology class $\alpha$. A detailed definition is given in Section \ref{background}.\\

In Riemannian geometry, the dimension of the space of geodesic segments connecting any two points can be defined using the space of Jacobi fields. In Section \ref{Jacobi} the ``Jacobi fields'' are defined and explicitly constructed, and the dimension of the space of geodesic segments is defined in Section \ref{proof}.\\

Curve complexes are in general locally infinite, so there can be infinitely many geodesic arcs connecting two vertices. In order to be able to prove theorems in a locally infinite complex, the concept of tightness was introduced in \cite{MasurandMinskyII} and modified in \cite{Bowditch3}. Subsection \ref{background} defines tightness for $\mathcal{C}(S,\alpha)$. It is a classic result from \cite{MasurandMinskyII} and \cite{Bowditch3} that there are only finitely many tight geodesics connecting any two verticies $m_1$ and $m_2$ in the complex of curves $\mathcal{C}(S)$. In $\mathcal{C}(S,\alpha)$, it also follows from the main theorem of this paper that there are finitely many tight geodesics connecting any two vertices; however, unlike in $\mathcal{C}(S)$, this is not a consequence of hyperbolicity, and geodesics do not fellow travel in $\mathcal{C}(S, \alpha)$, as demonstrated in Figure 13 of \cite{Me2}. \\

In \cite{Me2} an algorithm for constructing geodesics was given, which will be outlined briefly in Section \ref{background} for completeness. This paper develops the idea that the quantity called the ``overlap function'' used in this algorithm for constructing geodesics has strong parallels with a Morse function. Critical levels defined in Section \ref{proof} are analogues of conjugate points along geodesics in Riemannian geometry. The bounded topology of $S$ gives a uniform bound on the number of critical levels, from which the theorem follows:\\

\begin{thm}
\label{maintheorem}
Given any two vertices, $m_1$ and $m_2$ in $\mathcal{C}(S,\alpha)$, the space of tight geodesics connecting $m_1$ and $m_2$ has dimension less than $36\chi(S)^{2}$. 
\end{thm}


The \textit{Torelli group} $\mathcal{T}$ of $S$ is the subgroup of the mapping class group of $S$ that acts trivially on $H_{1}(S,\mathbb{Z})$. The complex $\mathcal{C}(S, \alpha)$ is a member of a family of complexes that generalise the complex of curves to study $\mathcal{T} $. For example, in \cite{BBM} to calculate cohomological properties of $\mathcal{T}$, in \cite{HM} to reprove a result of Birman-Powell about the generating set of the Torelli group of a surface with genus at least three, and in \cite{Chillout} to give a combinatorial description of a Torelli group invariant known as the Chillingworth class. Distances in these complexes are closely related to Seifert genuses of links in 3-manifolds, \cite{Me2}.\\

\textbf{Sublevel Projection.} The Masur-Minsky notion of subsurface projection is not directly applicable to many problems arising from studying $\mathcal{C}(S, \alpha)$. Questions relating to the way the Torelli group restricts to subsurfaces have already been shown to be central to understanding generating sets of the Torelli group, \cite{Putman2}. In Section \ref{sublevel} a notion analogous to subsurface projection from \cite{MasurandMinskyII} is defined by restricting to level sets of the overlap function, to which the ``projections'' are as rigid as possible. A distance formula analogous to that in \cite{MasurandMinskyII} follows from the finite number of critical levels and distance calculations in \cite{Me2}.\\



\subsection{Acknowlegements} This work was funded by a MOE AcRF-Tier 2 WBS grant Number R-146-000-143-112.

\section{Background and Notation.} \label{background} 
A \textit{curve} $c$ in $S$ is a piecewise smooth, injective map of $S^1$ into $S$ that is not null homotopic. A \textit{multicurve} is a union of pairwise disjoint curves on $S$. Let $\alpha$ be a primitive, nontrivial element of $H_{1}(S,\mathbb{Z})$. The \textit{complex of cycles}, $\mathcal{C}(S,\alpha)$, is a graph whose vertex set is the set of all isotopy classes of oriented multicurves in $S$ in the primitive homology class $\alpha$. There is an edge passing from $m_1$ to $m_2$ if $m_1$ and $m_2$ represent multicurves whose difference is isotopic to the oriented boundary of an embedded subsurface of $S$ with the subsurface orientation. Higher dimensional simplices can also be defined as in \cite{BBM} for example, however they will not be needed in this paper. The \textit{distance}, $d_{\mathcal{C}}(m_{1},m_{2})$, between $m_1$ and $m_2$ in $\mathcal{C}(S,\alpha)$ is defined to be the usual path metric, where all edges have length one.\\

\begin{rem}The assumption that edges of $\mathcal{C}(S,\alpha)$ represent embedded, consistently oriented subsurfaces is not necessary for Theorem \ref{maintheorem}, but makes many definitions and discussions considerably simpler. \end{rem}

Where this does not cause confusion, the same symbol will be used for vertices of $\mathcal{C}(S,\alpha)$ and corresponding multicurves on $S$. Also, multicurves will regularly be confused with the image in $S$ of a particular representative of the isotopy class.\\

The notation $m_{1}, \gamma_{1}, \gamma_{2}, \ldots, m_{2}$ will be used to denote a path $\gamma$ connecting the vertices $m_1$ and $m_2$ (the $\gamma_i$ are the vertices the path passes through).  \\

\textbf{Tightness}. Two multicurves $m_1$ and $m_2$ in general position are said to \textit{fill} $S$ if their complement in $S$ is a union of discs.\\

The notion of ``tightness'' was first defined in \cite{MasurandMinskyI} in order to prove theorems in a complex that is not locally finite. According to the variant of the definition in \cite{Bowditch3}, a path $c_{0}, c_{1}, \ldots, c_{n}$ in Harvey's complex of curves  $\mathcal{C}(S)$ was called \textit{tight} at the index $\left\{i\neq 0,n\right\}$ if every curve on the surface $S$ that crosses $c_i$ also crosses some element of $c_{i-1}\cup c_{i+1}$. Informally, this definition ensures that $c_{i}$ is contained within or on the boundary of the connected subspace of $S$ filled by $c_{i-1}\cup c_{i+1}$. Recall that (for $\mathcal{C}(S)$) any two multicurves representing vertices in $\mathcal{C}(S)$ separated by a distance at least three automatically fill $S$. It therefore automatically follows from the definition that $c_{i}$ is contained within or on the boundary of the connected subspace of $S$ filled by $c_{j}\cup c_k$, for all $j<i$ and $k>i$.\\

However, for $\mathcal{C}(S, \alpha)$, vertices separated by an arbitrarily large distance do not necessarily fill $S$, \cite{Me2}. In Example \ref{alternative}, the distance between $m_1$ and $m_2$ could be made arbitrarily large by increasing the number of bounding pair maps needed to map $m_1$ to $m_2$, without $m_1$ and $m_2$ filling the surface. For this reason, a path $\left\{\gamma_{1}, \ldots, \gamma_{n}\right\}$ in $\mathcal{C}(S, \left[\gamma_{1}\right])$ is defined to be \textit{tight} if, for every curve $c$ in $\gamma_i$, every curve on the surface $S$ that crosses $c$ also crosses some element of $\gamma_{j}\cup \gamma_{k}$, for all $j<i$ and $k>i$. This definition then rules out the possibility that a subpath of a tight geodesic enters a subsurface of $S$ that the two endpoints of the path do not enter.\\

From now on, all geodesic segments will be assumed to be tight.\\

Some background from \cite{Me2} on how to construct geodesics will be briefly repeated here.\\ 

The \textit{overlap function}, also denoted by the symbol $f(n)$, of a null homologous union of curves, $n$, is a locally constant, upper semi continuous, integer valued function defined on $S$ with minimum value zero. For any two points $x$ and $y$ in $S\setminus n$, $f(x)-f(y)$ is the algebraic intersection number of $n$ with an oriented arc with starting point $y$ and endpoint $x$. An important special case is the overlap function of the difference of two homologous multicurves, $m_{2}-m_1$. \\


The overlap function is not dependent on the choice of oriented arc, because the algebraic intersection number of any closed loop with $n$ is zero. It does however depend on the choice of representatives of the homotopy classes of curves. It will be assumed that the representatives of the homotopy classes are chosen so that the maximum, $M$, of the overlap function is as small as possible. When $n$ does not contain homotopic curves, it is sufficient to assume that the curves in $n$ are in general and minimal position.  For two homologous multicurves $m_1$ and $m_2$, the quantity $M$ will be called the \textit{homological distance}, $\delta(m_{1}, m_{2})$, between $m_1$ and $m_2$. \\

\begin{cor}[Corollary of Theorem 4 of \cite{Me2}]
Let $m_{1}$ and $m_{2}$ be two multicurves corresponding to vertices of $\mathcal{C}(S, \alpha)$. Then $d_{\mathcal{C}}(m_{1}, m_{2})=\delta(m_{1}, m_{2})$. 
\label{j}
\end{cor}

\textbf{Surgery along a horizontal arc}. Since both $S$ and $m_1$ are oriented, if $t(m_{1})$ is a tubular neighbourhood of $m_1$, $t(m_{1})\setminus m_{1}$ consists of two components; one of which can be said to be ``to the right'' of $m_1$ and the other ``to the left''. An arc of $m_{2}\cap(S\setminus m_{1})$ will be said to be \textit{vertical} if, for any tubular neighbourhood of $m_1$, the arc intersects both the component of $t(m_{1})\setminus m_{1}$ to the left of $m_1$ and the component to the right. If an arc of $m_{2}\cap(S\setminus m_{1})$ is not vertical, it will be said to be \textit{horizontal}. A horizontal arc can be either to the left of $m_1$ or to the right of $m_1$. Let $a$ be a horizontal arc with endpoints on a multicurve $m$. A tubular neighbourhood of $m\cup a$ has boundary consisting of a multicurve isotopic to $m$, and some other multicurve, call it $s_{a}(m)$. To \textit{surger} $m$ \textit{along} $a$ is to replace $m$ with $s_{a}(m)$. Surgering along a horizontal arc clearly does not change the homology class of a multicurve.\\

A surgery along a horizontal arc $a$ will be denoted by $s_a$. When talking about surgering along an arc, the implicit assumption is that the arc is horizontal.\\

It is known that all tight paths, geodesic or otherwise, connecting $m_1$ to $m_2$ within $\mathcal{C}(S,\alpha)$ can be constructed as follows: surger $m_1$ along some set of horizontal arcs  of $m_{2}\cap(S\setminus m_{1})$, and/or discard  a null homologous multicurve to obtain $\gamma_1$. Repeat with $\gamma_{1}$ in place of $m_1$ to obtain $\gamma_{2}$, etc. A proof can be found in \cite{Hatcher}.\\

If $a$ is an arc of $m_{2}\cap(S\setminus m_{1})$, it will be said to be \textit{homotopic} to another arc $b$ of $m_{2}\cap(S\setminus m_{1})$ if it can be homotoped onto $b$ by a homotopy that keeps the endpoints of $a$ on $m_2$.\\

The reason for calling arcs horizontal or vertical is illustrated in Figure \ref{fence}. The overlap function is larger on one vertex of a vertical arc than it is on the other, while a horizontal arc has both endpoints in the same level set. When the overlap function of $m_{2}-\gamma_{i}$ is restricted to $m_2$, the horizontal arcs represent local extrema. Informally, homotopy classes of horizontal arcs of $m_{2}\cap (S\setminus \gamma_{i})$ represent the choices available in constructing the next vertex, $\gamma_{i+1}$, along a tight geodesic segment connecting $\gamma_{i}$ to $m_2$.\\


\textbf{Middle paths}. Let $S_{max}$ be the subsurface of $S$ on which the overlap function of $m_{2}-m_1$ has its maximum and $S_{imax}$ the subsurface of $S$ on which the overlap function of $m_{2}-\gamma_i$ has its maximum. Similarly for $S_{min}$ and $S_{imin}$. Also let $S_{a\leq f \leq b}$ be the subsurface of $S$  on which $a\leq f(m_{2}-m_{1})\leq b$. The boundary of  $S_{max}$ is a union of horizontal arcs of $m_{2}\cap (S\setminus m_{1})$ to the right of $m_1$ and horizontal arcs of $m_{1}\cap (S\setminus m_{2})$ to the left of $m_2$. It is not hard to check that surgering $m_1$ along the arcs of $m_{2}\cap (S\setminus m_{1})$ on the boundary of $S_{max}$ gives a multicurve $\gamma_{1}$ and a curve $-\partial S_{max}$, where $\delta(\gamma_{1}, m_{2})=\delta(m_{1}, m_{2})-1$ and the vertices $\gamma_{1}$ and $m_1$ are connected by an edge. Construct $\gamma_{2}$ in the same way, but with $S_{1max}$ instead of $S_{max}$ and $\gamma_{1}$ in place of $m_1$, similarly for $\gamma_{3}$, etc. A geodesic constructed in this way will be called a \textit{middle path}.  \\

\textbf{Critical levels and level sets.} If $\gamma_{i}$ is a vertex on a middle path, informally, a critical level should be thought of as a value of $i$ for which the level set $S_{M-i \leq f}$ is ``different'' from the previous level set $S_{M-i+1 \leq f}$. By different, is meant either the topology, or the number of edges on the boundary of the level set changes. The critical levels along geodesic segments are therefore closely related to local extrema or saddles of the overlap function. When trying to make this notion precise, there are some technicalities involved, especially for paths that are not middle paths, so a somewhat different approach will be taken in Section \ref{proof}. \\

Usually, a Morse theory is set up to compute a homology theory. It is not clear what the analogue, if any, of a homology theory might be in this case. Path construction in $\mathcal{C}(S,\alpha)$ has a lot of similarities with tracing out the stable or unstable manifolds coming from the local extrema of the overlap function of $m_{2}-m_{1}$. The finite dimensions of the space of geodesics might then be thought of as coming from the choices about the order in which different stable or unstable manifolds are traced out.\\



\textbf{Labelling geodesic segments and surgeries.} In this paper, surgeries will be denoted by listing the elements of a set of arcs along which a multicurve is surgered. The superscripts on the arcs in the set determine the multicurve along which the surgery is performed, and the subscripts label the elements in the set. For geodesic segments in a one parameter family, the superscripts will denote the element of the family, and the subscripts determine the vertex of a geodesic segment.\\

\subsection{Independent Surgeries} \label{independent}When making statements about how to perturb the geodesic segment $m_{1}, \gamma_{1}, \ldots, m_{2}$, it is necessary to have a concept of what surgeries are equivalent to or dependent on each other. In order to understand this, we first need a notation for the smallest subsurface inside which a multicurve is altered by a surgery and the subsequent isotopy to put it in minimal and general position with $m_2$.\\

A homotopy class of arcs with representative $a$ determines a rectangle $R(a)$ in $S$, as shown in Figure \ref{chfour1}. The ``short sides'' of $R(a)$ are arcs in the homotopy class.\\

\begin{figure}
\begin{center}
\def\svgwidth{12cm}
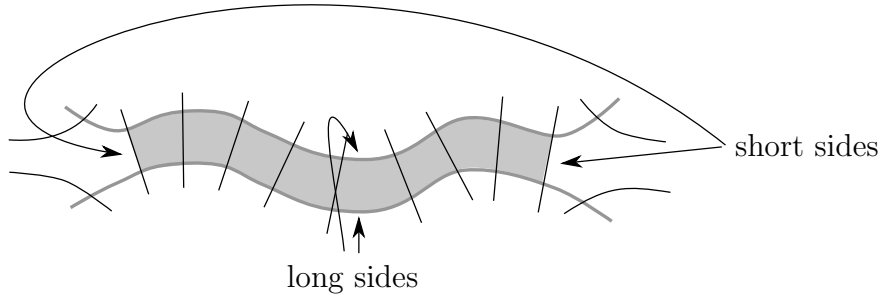
\caption{The rectangle representing a homotopy class of arcs. Diagram taken from \cite{Me2}.}
\label{chfour1}
\end{center}
\end{figure}

When constructing a path in $\mathcal{C}(S, \alpha)$, a surgery along an arc $a$ is \textit{independent} of a surgery along an arc $b$ if $R(a)\cup N(a)$ and $R(b)\cup N(b)$ are disjoint, where $N(a)$ is the null homologous submulticurve (if any) discarded after surgering along $a$, and $N(b)$ the null homologous submulticurve (if any) discarded after surgering along $b$.\\

When $\{a_{j}^{i+1}\}$ is a collection of horizontal arcs of $m_{2}\cap (S\setminus \gamma_{i})$, it could happen that all but one, $g_k$, of the arcs of $m_{2}\cap(S\setminus \gamma_{i})$ on the boundary of a polygon $G_k$ in $S\setminus (m_{2}-\gamma_{i})$ are homotopic to one of the arcs $\{a_{j}^{i+1}\}$. The surgeries corresponding to  $\{a_{j}^{i+1}\}$ are independent of the surgeries corresponding to the arcs $\{b_{l}^{m}\}$ if $\cup_{j}R(a_{j}^{i+1})\cup_{k}G_{k}\cup_{k}R(g_{k})\cup N(\{a_{j}^{i+1}\})$ is disjoint from $\cup_{l}R(b_{l}^{m})\cup_{n}G_{n}\cup_{n}R(b_{n})\cup N(\{b_{l}^{m}\})$.\\

\textbf{Equivalent Surgeries}. It can happen that two independent surgeries, followed by discarding different null homologous submulticurves can give the same result up to isotopy. Two such surgeries will be said to be \textit{equivalent}. An example of this can be found in Example \ref{alternative}. The curve $\gamma_{9}$ is obtained from $m_2$ by applying a bounding pair map four times. There are two horizontal arcs of $m_{2}\cap (S\setminus\gamma_{9})$, and surgering along either of them results in untwisting one pair of twists.\\

\section{Jacobi Fields}\label{Jacobi}

In order to motivate the definition of Jacobi fields, it helps to have a few simple examples in mind. These examples are given in Subsection \ref{examples}. Subsection \ref{Jacobi} then defines and constructs one parameter families and their associated Jacobi fields. Finally, Subsection \ref{linearcombinations} makes rigorous the notion of a linear combination of Jacobi fields.\\

\subsection{Examples}\label{examples}

\begin{figure}
\centering
\includegraphics[width=\textwidth]{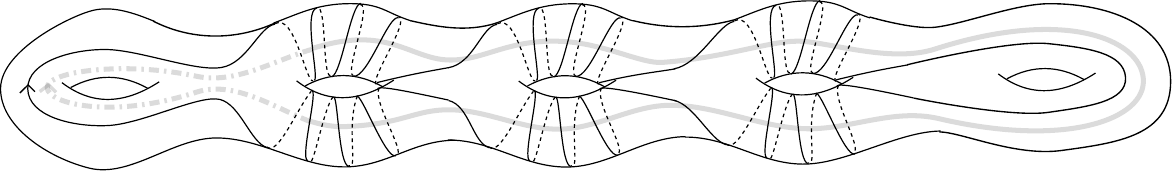}
\caption{The curves $m_1$ and $m_2$ (grey) from Example \ref{alternative}. The arc $a_1$ is the fat dotted grey line. }
\label{bigheaptwist}
\end{figure}

\begin{figure}
\centering
\includegraphics[width=\textwidth]{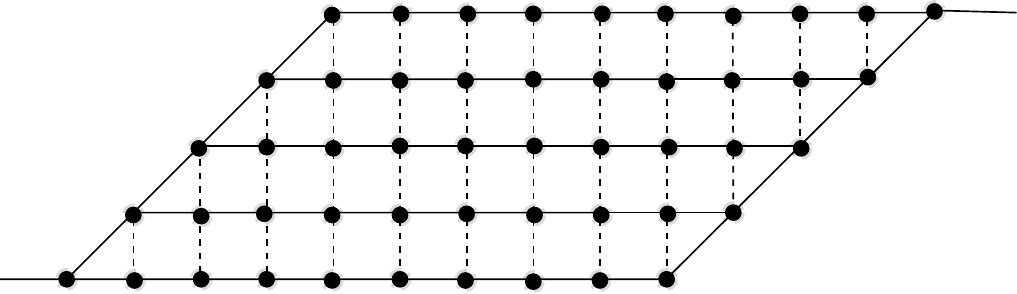}
\caption{A one parameter family of geodesics in $\mathcal{C}(S,[m_{1}])$ from Example \ref{alternative}. Geodesic segments in the one parameter family are represented by solid lines, other edges in the graph by dotted lines.}
\label{subgraph}
\end{figure}

\begin{ex}[Alternative Surgeries] \label{alternative}

The curves $m_1$ and $m_2$ are shown in Figure \ref{bigheaptwist}. The geodesic $m_{1}, \gamma_{1},\ldots,\gamma_{12}, m_{2}$ is constructed by first surgering $m_1$ along the arc $a_1$ and discarding a resulting null homologous multicurve to obtain $\gamma_1$. The curve $\gamma_1$ has one fewer of the pairs of twists furthest to the left. The multicurve $\gamma_{2}$ is obtained similarly by surgering along an arc $v_{1}\circ a_{1} \circ v_{2}$, where $v_1$ and $v_2$ are arcs of $m_{2}\cap (S\setminus m_{1})$ to either side of $a_1$. This surgery undoes the next leftmost pair of twists. The curves $\gamma_{3}$ and $\gamma_4$ are obtained similarly. Once we get to $\gamma_{5}$, we start unwinding pairs of twists inside the genus one subsurface to the right of the first subsurface. Last of all, the twists inside the rightmost subsurface are undone.\\

The decisions involved in constructing $m_{1}, \gamma_{1},\ldots,\gamma_{12}, m_{2}$ were completely arbitrary. For example, we could construct a family of geodesic segments $m_{1}, \gamma_{1}^{k},\ldots,\gamma_{12}^{k}, m_{2}$, as follows: $m_{1}, \gamma_{1}^{k},\ldots,\gamma_{12}^{k}, m_{2}$ is the geodesic segment obtained by first untwisting $k$ twists, working from right to left, and then untwisting from left to right. This family of geodesic segments in $\mathcal{C}(S,\alpha)$ is depicted in Figure \ref{subgraph}.\\


\end{ex}

In the previous example, it is the assumption that paths are tight that rules out the possibility of untwisting in the middle first.\\

\begin{figure}
\begin{center}
\def\svgwidth{\textwidth}
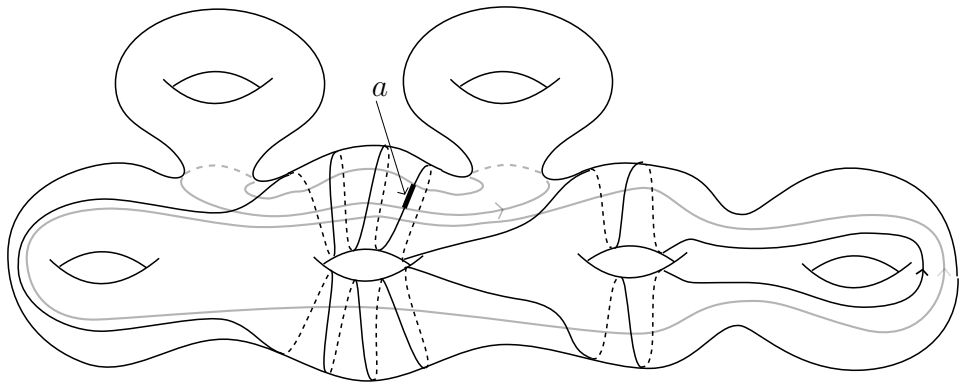
\caption{The multicurve $m_1$ is shown in grey.}
\end{center}
\label{crabfigure}
\end{figure} 

\begin{figure}
\centering
\includegraphics[width=\textwidth]{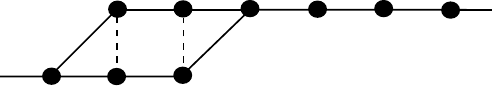}
\caption{Another one parameter family of geodesics in $\mathcal{C}(S,[m_{1}])$. Geodesic segments are represented by solid lines.}
\label{subgraph2}
\end{figure}

\begin{ex}[Optional Surgeries] \label{crab}

In this example, the geodesic $m_{1}, \gamma_{1}, \ldots, \gamma_{5}, m_{2}$ is constructed by untwisting from right to left in Figure 4. When constructing $\gamma_{1}$, in addition, we might also have surgered along the arc $a$ shown. If we do not do this, at the very latest, $\gamma_{2}$ has to be surgered along a set of arcs including $a$ to obtain $\gamma_3$. This gives the (small) one parameter family $m_{1}, \gamma_{1}^{1}, \ldots, m_{2}$ depicted in Figure \ref{subgraph2}.

\end{ex}

\subsection{One parameter families and Jacobi fields.}\label{Jacobi}
In this subsection, different ways of constructing one parameter families of geodesic segments will be discussed. The main difficultly is in understanding the circumstances under which these constructions can be applied without causing contradictions with path construction on some other subsegment. The one parameter families are used to define ``Jacobi fields''. Before defining one parameter families, it is necessary to establish a canonical choice of isotopy classes of multicurves, so as to be able to identify arcs of $m_{2}\cap(S\setminus \gamma_{i})$ for different values of $i$. \\

\textbf{Choices of Isotopy Classes}. Put $m_1$ and $m_2$ in general and minimal position. The representative of the isotopy class $\gamma_1$ is then obtained as follows: first perform the surgeries corresponding to $\{a_{i}^{1}\}$ on $m_1$. Any part of the resulting multicurve either coincides with $m_1$ outside of an $\epsilon$-neighbourhood of $\cup_{i}R(a_{i}^{1})\cup_{j}G_{j}\cup_{j}R(g_{j})$, or if it has become part of a null homologous submulticurve, it might have been discarded. The multicurve $\gamma_2$ is then obtained by performing the surgeries corresponding to $\{a_{i}^{2}\}$ on this representative of the isotopy class $\gamma_1$. Any part of the resulting multicurve outside of an $\epsilon$-neighbourhood of $\cup_{i}R(a_{i}^{2})\cup_{j}G_{j}\cup_{j}R(g_{j})$ either coincides with $\gamma_{1}$ or is discarded, etc.\\

Jacobi fields come about in a few different ways; from optional surgeries, alternative surgeries or choices about null homologous submulticurves. First of all Jacobi fields coming from optional surgeries will be defined. \\

\textbf{Optional Surgeries}. Let $a$ be a horizontal arc of $m_{2}\cap(S\setminus \gamma_{i})$, that defines an \textit{optional} surgery in the following sense: $s_a$ is independent of the surgeries along the set of arcs $\{a_{j}^{i+1}\}$ performed on $\gamma_{i}$ to obtain $\gamma_{i+1}$. In addition, surgering along the set of arcs $\{a\}\cup\{a_{j}^{i+1}\}$ determines an edge of $\mathcal{C}(S,\alpha)$.\\

In Example \ref{crab}, there was a subinterval of $\gamma$ along which $s_a$ determined an optional surgery. Further along $\gamma$ at vertex $\gamma_{2}$, $s_a$ was one of the surgeries performed to obtain $\gamma_{3}$. As a result, the path $\gamma^{1}$ and $\gamma$ then converged on vertex $\gamma_3$. If $s_a$ is not equivalent to a surgery applied to one of the multicurves $\{\gamma_{i}\}$ somewhere along $\gamma$, then there is necessarily a surgery $s_b$ applied to $\gamma_{j}$ for some $j$, where the arc $b$ has one or more endpoints in common with $a$. The surgery $s_a$ is then no longer defined on $\gamma_{i}$, for $j<i$, and tightness rules out the possibility of applying $s_{a}^{-1}$ to $\gamma_{i}^{1}$ for $i<j$. It is then not clear how $\gamma_{i}^{1}$ should be related to $\gamma_{i}$ for $j<i$. In summary - the surgery $s_a$ does not determine a one parameter family unless it is equivalent to a surgery that is actually performed somewhere along $\gamma$. \\

\textbf{One parameter families}. Let $I_a$ be the largest subinterval of $m_{1}, \gamma_{1}, \gamma_{2},\ldots, m_2$ on which the arc $a$ determines an optional surgery on the preceding multicurve. By assumption the last vertex of $I_a$ is the vertex of $\gamma$ that is surgered along the arc $a$, for the reasons explained in the previous paragraph.  For $\gamma_{i}$ in $I_a$, let $\gamma_{i}^{1}$ be obtained from $\gamma_{i}$ by applying $s_a$. The vertex $\gamma_{i}^{1}$ coincides with $\gamma_i$ outside of $I_a$. The geodesic segment $m_{1}, \gamma_{1}^{1}, \gamma_{2}^{1},\ldots, m_2$ is the first element of the one parameter family above $\gamma$. \\

\begin{figure}
\begin{center}
\def\svgwidth{12cm}
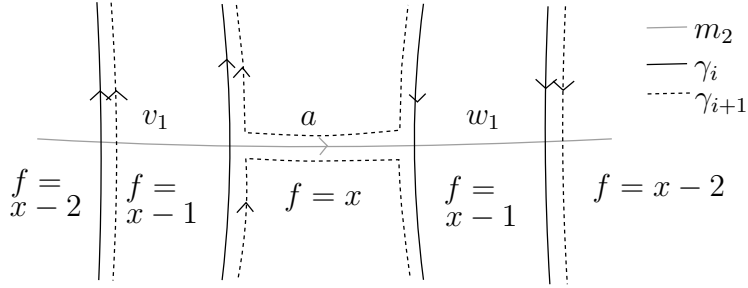
\caption{Consecutive surgeries.}
\label{fence}
\end{center}
\end{figure}

After surgering along a horizontal arc $a$ of $m_{2}\cap (S\setminus \gamma_{i})$, suppose a new horizontal arc, $v\circ a\circ w$ is created, as shown in Figure \ref{fence}. Surgering $\gamma_{i+1}$ along $v_{1}\circ a \circ w_{1}$ will be thought of as as being the most obvious continuation of the surgery along $a$.\\

To construct the second element of the one parameter family above $\gamma$, call it $\gamma^{2}$, let $v_{1}\circ a \circ w_1$ be the arc of $m_{2}\cap (S\setminus \gamma_{i}^{1})$ obtained by concatenating $a$ with arcs $v_1$ and $w_1$ of $m_{2}\cap (S\setminus \gamma_{i})$ on either side of it. If $v_{1}\circ a \circ w_1$  is not a horizontal arc of $m_{2}\cap(S\setminus \gamma_{i}^{1})$ that represents an optional surgery for some $i$, we are finished. Otherwise, $m_{1}, \gamma_{1}^{2}, \gamma_{2}^{2},\ldots, m_2$ is constructed from $m_{1}, \gamma_{1}^{1}, \gamma_{2}^{1},\ldots, m_2$ analogously to the way $m_{1}, \gamma_{1}^{1}, \gamma_{2}^{1},\ldots, m_2$ was constructed from $m_{1}, \gamma_{1}, \gamma_{2},\ldots, m_2$. Let $n$ be the natural number such that the one parameter family can not be extended past $m_{1}, \gamma_{1}^{n}, \gamma_{2}^{n},\ldots, m_2$.\\

A \textit{Jacobi field}, $J(a,\gamma)$, is associated with the one parameter family as follows: The \textit{magnitude} of  $J(a,\gamma)$ at vertex $\gamma_i$ is equal to the maximum of $d(\gamma_{i}, \gamma_{i}^{k})$, for $1\leq k\leq n$. The \textit{support} of $J(a,\gamma)$ is the largest subpath of $m_{1}, \gamma_{1}, \gamma_{2},\ldots, m_2$ for which $\gamma_{i}^{k}\neq \gamma_{i}$ for some $k$. Along the support of $J(a,\gamma)$, the \textit{direction} of $J(a,\gamma)$ at the vertex $\gamma_i$ is parallel to the edge passing from $\gamma_i$ to the nearest vertex on a neighbouring geodesic segment in the one parameter family.\\

\textbf{Alternative Surgeries}.  Suppose $\gamma_{h+1}$ was constructed from $\gamma_h$ by surgering along arcs $\{a_{j}^{h+1}\}$ and possibly discarding a null homologous submulticurve $N(h+1)$, but that, alternatively, a geodesic segment could have been constructed by surgering along the arcs $\{c_{j}^{h+1}\}$ instead of some of the $\{a_{j}^{h+1}\}$ (call this set of arcs $\{b_{j}^{h+1}\}$) and possibly discarding a null homologous submulticurve $N(c,h+1)$. Whenever the set $\{c_{j}^{h+1}\}$ could not have been replaced by a smaller subset, the surgeries along $\{c_{j}^{h+1}\}$ will be called \textit{alternative surgeries}.\\

As when constructing one parameter families coming from optional surgeries, it is necessary to rule out the possibility that an alternative surgery is incompatible with another surgery further along $\gamma$. The easiest way of understanding what is meant by ``incompatible'' in this context, is to try to construct a one parameter family, and see what conditions are needed to do this successfully.\\ 


For ease of construction, it will first be assumed that $\{b_{j}^{h+1}\}$ is all of $\{a_{j}^{h+1}\}$. To start off with, also consider an example such as Example \ref{alternative}, where for $\gamma_{h}$ we have an alternative set of surgeries $\{c_{j}^{h+1}\}$. We would like to construct a one parameter family as follows: $\gamma_{i}^{1}$ coincides with $\gamma_i$ for $i\leq h$. The multicurve $\gamma_{h+1}^{1}$ is obtained from $\gamma_{h}$ by surgering along $\{c_{j}^{h+1}\}$ and possibly discarding a null homologous submulticurve $N(c,h+1)$. Unless $\{a_{j}^{i}\}$ is equivalent to $\{c_{j}^{h+1}\}$, for $h+1<i$, $\gamma_{i+1}^{1}$ is obtained by surgering $\gamma_{i}^{1}$ along $\{a_{j}^{i}\}$. If $\{a_{j}^{i}\}$ for $h+1<i$ is equivalent to $\{c_{j}^{h+1}\}$, denote this value of $i$ by $i^{*}$, then $\gamma^{1}$ coincides with $\gamma$ from $\gamma_{i^{*}+1}$ onwards. The next geodesic, $\gamma^{2}$, in the one parameter family is constructed analogously, with the alternative set of surgeries of the form $\{v_{j}\circ c_{j}^{h+1}\circ w_{j}\}$, where possible. Similarly for $\gamma^{3}$, etc. \\

To make this construction work, it is sufficient that the surgeries along the arcs $\{a_{j}^{i}\}$ for $h+1\leq i \leq i^{*}$ are independent of those along $\{c_{j}^{h+1}\}$. As Example \ref{replace} illustrates, this is a stronger condition than is desirable to impose in general.\\

\begin{figure}
\begin{center}
\def\svgwidth{12cm}
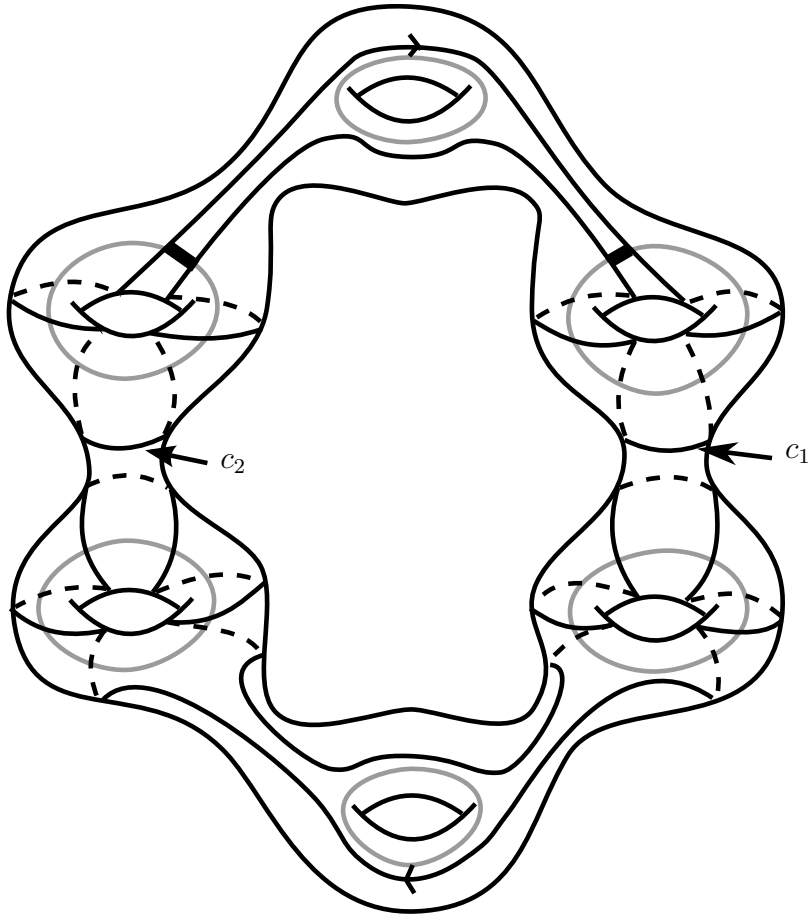
\caption{The multicurve $m_2$ is shown in grey.}
\label{diagonal}
\end{center}
\end{figure}

\begin{ex}
\label{replace}
The curves $m_1$ and $m_2$ are drawn in Figure \ref{diagonal}. The geodesic segment $\gamma$ is constructed by unwinding the twists from top to bottom, starting with a surgery along the thick black arcs $a_{1}^{1}$ and $a_{2}^{1}$ in Figure \ref{diagonal}. The geodesic segment $\gamma^{1}$ is constructed by first surgering along the arcs $c_{1}$ and $c_{2}$ shown in Figure \ref{diagonal}, and then along $\{a_{1}^{1}, a_{2}^{1}\}$, $\{a_{1}^{3}, a_{2}^{3}\}$, $\{a_{1}^{4}, a_{2}^{4}\}$ etc. The surgery along $\{c_{1}, c_{2}\}$ is not independent of the surgery along $\{a_{1}^{2}, a_{2}^{2}\}$; the arcs have the same endpoints, but is intended to replace the surgery along $\{a_{1}^{2}, a_{2}^{2}\}$. 
\end{ex}

Going back to the previous problem of constructing one parameter families for alternative surgeries, the alternative surgeries along the arcs $\{c_{j}^{h+1}\}$ will be said to \textit{replace} the surgeries along the arcs $\{a_{j}^{k+1}\}$ if the sets of arcs, $\{c_{j}^{h+1}\}$ and $\{a_{j}^{k+1}\}$, representing the surgeries can be chosen to have the same endpoints. If $\{c_{j}^{h+1}\}$ replaces the surgeries along the set of arcs $\{a_{j}^{k+1}\}$, let $\gamma^{1}$ be the geodesic segment that coincides with $\gamma$ for vertices $\gamma_{i}$, with $i\leq h$, $\gamma_{h+1}^{1}$ is obtained from $\gamma_{h}^{1}$ by surgering along $\{c_{j}^{h+1}\}$, for $h+2\leq i \leq k$, $\gamma_{i+1}^{1}$ is obtained from $\gamma_{i}^{1}$ by surgering along $\{a_{j}^{i}\}$, and for $k<i$, $\gamma_{i+1}^{1}$ is obtained from $\gamma_{i}^{1}$ by surgering along $\{a_{j}^{i+1}\}$. Further elements of the one parameter family, if any, are obtained similarly.\\

When for $h<k$, the surgeries along $\{c_{j}^{h+1}\}$ replace the surgeries along $\{a_{j}^{k+1}\}$, what conditions ensure that the alternative surgery determines a one parameter family? It is necessary and sufficient to assume that the surgeries along the arcs $\{a_{j}^{i}\}$ for $h+1\leq i \leq i^{*}$ are independent of those along $\{c_{j}^{h+1}\}$, where $i^{*}=k$. \\



Now suppose the more general case; the set of arcs $\{c_{j}^{h+1}\}$ determine alternative surgeries to a subset $\{b_{j}^{h+1}\}$ of the arcs $\{a_{j}^{h+1}\}$. One possibility is that the arcs $\{a_{j}^{h+1}\}\setminus \{b_{j}^{h+1}\}$ determine optional surgeries. If so, a one parameter family is constructed as in the previous few paragraphs, with the added constraint that the arcs $\{a_{j}^{h+1}\}\setminus \{b_{j}^{h+1}\}$ determine optional surgeries on each of the geodesics in the one parameter family. \\

When the arcs $\{a_{j}^{h+1}\}\setminus \{b_{j}^{h+1}\}$ do not determine optional surgeries, construct $\gamma^{1}$ as follows, where possible: For $i\leq h$, $\gamma_{i}^{1}$ coincides with $\gamma_{i}$. The multicurve $\gamma_{h+1}^{1}$ is constructed by surgering $\gamma_{h}^{1}$ along $\{c_{j}^{h+1}\}\cup (\{a_{j}^{h+1}\}\setminus \{b_{j}^{h+1}\})$. The sets of arcs $\{c_{j}^{h+1}\}$ and $\{b_{j}^{h+1}\}$ are both necessarily on the same side of $\gamma_h$, because surgering $\gamma_{i}$ along arcs that are not all on the same side of $\gamma_{i}$ does not determine an edge of $\mathcal{C}(S,[\gamma_{i}])$. Now assume there is a smallest possible subset $\{d_{j}^{h+2}\}$ of $\{a_{j}^{h+2}\}$ such that surgering $\gamma_{h+1}^{1}$ along $\{d_{j}^{h+2}\}\cup \{b_{j}^{h+1}\}$ gives a multicurve $\gamma_{h+2}^{1}$, where $\gamma_{h+2}^{1}$ and $\gamma_{h+1}^{1}$ are connected by an edge, and $\gamma_{h+2}^{1}$ is one unit closer to $m_1$ than $\gamma_{h+1}^{1}$. The multicurve $\gamma_{h+3}^{1}$ is constructed similarly, etc. up to $\gamma_{i^{*}}^{l}$, where the homotopy classes with representatives $\{a_{j}^{i^{*}+1}\}$ have endpoints in common with the homotopy classes of arcs with representatives $\{c_{j}^{h+1}\}$.\\

To make this construction work, the arcs $\{d_{j}^{h+2}\}$ are necessarily on the same side of $\gamma_{h+1}^{1}$ as the arcs $\{a_{j}^{h+1}\}$ are on $\gamma_{h}$, from which it follows that all the arcs $\{a_{j}^{h+2}\}$ are on the same side of $\gamma_{h+1}$ as the arcs $\{a_{j}^{h+1}\}$ are on $\gamma_h$. By induction, this is true up to and including $\{a_{j}^{i^{*}+1}\}$.\\

Arcs of $m_{2}\cap(S\setminus \gamma_{h})$ on the same side of $\gamma_h$ and with at least one endpoint in common necessarily coincide. For this reason, the set of representatives of homotopy classes of arcs, $\{a_{j}^{i^{*}+1}\}$, can be chosen to have at least one arc in common with $\{c_{j}^{h+1}\}$. It can be assumed without loss of generality that the arcs $\{c_{j}^{h+1}\}$ are a subset of $\{a_{j}^{i^{*}+1}\}$, because otherwise it will be seen in the next subsection that the resulting Jacobi fields can be obtained as linear combinations of Jacobi fields coming from one parameter families for which the arcs $\{c_{j}^{h+1}\}$ are a subset of $\{a_{j}^{i^{*}+1}\}$. For $i^{*}< i$, $\gamma_{i}^{1}$ therefore coincides with $\gamma_{i}$. Further elements in the one parameter family are also constructed inductively, as for the previous cases studied.\\

To make the previous construction work, we have already seen that the arcs $\{c_{j}^{h+1}\}$ necessarily represent surgeries that are actually performed on some $\gamma_i$; the one parameter family just changes the order in which commuting surgeries are performed. Also, the arcs $\{a_{j}^{i+1}\}$ all have to be on the same side of their respective multicurves for $h\leq i\leq i^{*}$ and it must be possible to find arcs $\{d_{j}^{i+2}\}$ for $h\leq i \leq i^{*}-1$. The existence of arcs $\{d_{j}^{i+2}\}$ is a strong assumption to make, for example, it rules out the possibility that $\{a_{j}^{h+2}\}$ consists of a single arc.\\

\textbf{Null homologous submulticurves}. Suppose $\gamma_{i}$ has null homologous submulticurves $N_{1}, N_{2}, \ldots N_m$ in the oriented isotopy class $n$. It can happen that discarding one of these submulticurves from $\gamma_{i}$ decreases the homological distance from $m_{2}$. When this happens, $n$ will be called \textit{nonperipheral} in $\gamma_{i}$, otherwise $n$ is \textit{peripheral} in $\gamma_{i}$. In Example \ref{crab}, the multicurves $\gamma_{1}$ and $\gamma_{1}^{1}$ have peripheral null homologous submulticurves.\\

When the subsurface bounded by $N_1$ is disjoint from the subsurface bounded by $\gamma_{i+1}-\gamma_{i}$, discarding $N_1$ can be treated as an optional surgery. A one parameter family is obtained when, for some $i<k$, the null homologous submulticurve $N_1$ is discarded from $\gamma_{k}$. A second geodesic segment in the one parameter family is constructed by taking the most obvious continuation of discarding $N_1$, namely discarding $N_2$, etc.\\

Similarly, if $n$ is nonperipheral, discarding $N_1$ is analogous to an alternative surgery. Discarding $N_1$ commutes with any set of surgeries along arcs whose endpoints are not on $N_{1}$, and it is clear how to construct a one parameter family by changing the order of commutative operations. Otherwise, let $\gamma$ be the path with all the $N_{i}$s discarded as soon as possible. The numbering of the $\{N_{i}\}$ is assumed to reflect the order in which the multicurves are discarded. Let $\gamma^{1}$ be a geodesic segment for which all the $\{N_{i}\}$ but $N_1$ are discarded, and let $\gamma_{k+1}^{1}$ be the first vertex of $\gamma^{1}$ that does not coincide with the corresponding vertex on $\gamma$. Let $\{c_{j}^{k}\}$ be the set of arcs along which $\gamma_{k}^{1}$ is surgered to obtain $\gamma_{k+1}^{1}$. The set $\{c_{j}^{k}\}$ is obtained by modifying $\{a_{j}^{k+1}\}$ as follows: if $a_{j}^{k+1}$ is to the right of $\gamma_{k}$, whenever $a_{j}^{k+1}$ intersects $N_1$, replace $a_{j}^{k+1}$ with the intersection of $a_{j}^{k+1}$ with the subsurface of $S$ to the right of $N_1$. Since $N_1$ is null homologous, this is necessarily a set of horizontal arcs. If $a_{j}^{k+1}$ does not intersect $N_1$, leave it unchanged. When the arc $a_{j}^{k+1}$ is to the left of $\gamma_{k}$, replace it with the intersection of $a_{j}^{k+1}$ with the subsurface of $S$ to the left of $N_1$. Similarly, $\{c_{j}^{k+1}\}$ is obtained by modifying the set $\{a_{j}^{k+2}\}$ as follows: if $a_{j}^{k+2}$ is to the right of $\gamma_{k+1}$, replace $a_{j}^{k+2}$ by the set of arcs $a_{j}^{k+2}\cap (S\setminus \gamma_{k+1}^{1})$ (we are assuming the standard choice of isotopy class) to the right of $\gamma_{k+1}^{1}$, etc. The last vertex of $\gamma^1$ before $m_2$, call it $\gamma_{l}$, is constructed by surgering along all arcs $a_{j}^{l}$ that were not disjoint from $\gamma_{l-1}^{1}$.\\

In the previous paragraph, it can be assumed that $N_1$ cuts $S$ into subsurfaces, one of which is to the left of $N_1$ and one of which is to the right of $N_1$. If $N_1$ were a set of nested, null homologous multicurves for which this is not true, discarding $N_1$ would not define an edge of $\mathcal{C}(S, \alpha)$.\\

The geodesic segment $\gamma^{2}$ is obtained similarly from $\gamma^{1}$, by not discarding the null homologous multicurve $N_2$, etc. \\

\textbf{Restrictions of Jacobi fields.} The \textit{restriction} of a Jacobi field $J(a, \gamma)$ can be defined. This is done by constructing a one parameter family $\gamma, \gamma_{r}^{1}, \gamma_{r}^{2}, \ldots$ contained within the one parameter family $\gamma, \gamma^{1}, \gamma^{2}, \ldots$ to which $J(a, \gamma)$ is tangent. The vertices of the geodesic segments $\gamma_{r}^{i}$ are all vertices on the geodesic segments $\gamma, \gamma^{1}, \gamma^{2}, \ldots$. Taking a restriction of a Jacobi field is the same thing as multiplying by a scalar field $\phi\in \mathbb{Q}$, $0\leq \phi \leq 1$ such that $\phi J(a,\gamma)$ determines a valid one parameter family. \\

Restrictions of Jacobi fields can interpolate between geodesic segments, at least one of which is constructed in a seemingly random way. Consider for example a path $\delta$ with $m_1$ and $m_2$ as in Example \ref{alternative}, where $\delta$ is constructed by surgering along arcs to the left or to the right in a random way. A restriction of the Jacobi field tangent to the one parameter family described in the example interpolates between $\gamma$ and $\delta$.\\


\begin{rem}The definitions of one parameter families are symmetric in $m_1$ and $m_2$, but the directions of the Jacobi fields reverse when $m_1$ and $m_2$ are interchanged. To understand why this is so, note that surgering along a horizontal arc has an inverse. When $m_1$ and $m_2$ are interchanged, this has the effect of exchanging a surgery with its inverse. It follows that the same definition of one parameter family corresponding to the optional surgery $s_a$, when applied to $m_{2}, \gamma_{j}^{n}, \gamma_{j-1}^{n},\ldots,m_{1}$ in place of $m_{2}, \gamma_{j}, \gamma_{j-1}, \ldots, m_{1}$, and $s_{a}^{-1}$ in place of $s_a$, gives the same family. Exactly the same is true for Jacobi fields arising in other ways. The Figures \ref{subgraph} and \ref{subgraph2} were drawn in such a way as to highlight this symmetry.
\end{rem}

\subsection{Linear Combinations of Jacobi Fields} \label{linearcombinations}We would like to be able to describe all geodesics connecting $m_1$ and $m_2$ by taking linear combinations of Jacobi fields. However, it is necessary to make sure that the linear combination determines a valid set of deformations within one parameter families. There are constraints to check, and it is necessary to make sense of what it means to add Jacobi fields representing surgeries that are not independent. \\

The constraints are that edges can only connect disjoint multicurves, and an edge can only connect two multicurves whose difference is an embedded, consistently oriented subsurface of $S$.\\


To add two Jacobi fields with the same direction, whenever this gives another valid Jacobi field, we simply add the magnitudes and leave the direction unchanged.\\

The sum of two Jacobi fields $J(a,\gamma)$ and $J(b,\delta)$, where defined, should be thought of as a recipe for moving within two one parameter families. First, $J(a,\gamma)$ determines a deformation of the geodesic $\gamma$ within a one parameter family to obtain a geodesic $\gamma^{k}$. When $\gamma^{k}=\delta$, the second Jacobi field gives a recipe for a further deformation within a one parameter family of $\gamma^{k}$. \\

Subtraction of a Jacobi field $J$ is defined as the inverse of addition, i.e. a deformation within a one parameter family in the direction opposite to that determined by $J$.\\

Linear combinations of Jacobi fields do not necessarily represent Jacobi fields, because there may not be one parameter families to which the linear combination is tangent. It is necessary to consider noncommutative linear combinations in order to describe the entire space of geodesic segments connecting two vertices. When two Jacobi fields along $\gamma$ commute, for example, they have disjoint support, represent independent surgeries or are parallel, by abuse of notation their linear combination will be called a \textit{linear combination of two Jacobi fields along} $\gamma$.\\

As an example of a linear combination, let $J(n_{1},\gamma)$ and $J(n_{2},\gamma)$ be Jacobi fields that arise from discarding non peripheral null homologous multicurves in the isotopy classes $n_1$ and $n_2$, respectively. Suppose also $\gamma_{i}$ is in the intersection of the support of $J(n_{1},\gamma)$ and $J(n_{2},\gamma)$, and the interior of a multicurve in the isotopy class $n_{1}-n_{2}$ is disjoint from the interior of $\gamma_{i+1}-\gamma_i$. Then $J(n_{1},\gamma)$ is a linear combination of $J(n_{1}-n_{2},\gamma)$ and $J(n_{2},\gamma)$.\\

\begin{defn}[The dimension of the space of Jacobi fields along a geodesic segment]
The dimension of the space of Jacobi fields along a geodesic segment $\gamma$ is the smallest possible number of Jacobi fields along $\gamma$ in a set $\mathcal{J}$, such that any Jacobi field along $\gamma$ can be written as a linear combination of elements of $\mathcal{J}$.
\end{defn}

\section{Proof of Theorem \ref{maintheorem}}\label{proof}
To start off with, it will be shown that the Jacobi fields determine the entire space of geodesic segments in some sense. After this, the dimension of the space of geodesics will be defined, and Theorem \ref{maintheorem} proven. \\

\begin{defn}[The subspace of geodesic segments spanned by a set $\mathcal{J}$ of Jacobi fields]

Given two geodesic segments connecting $m_1$ and $m_2$, call them $\delta$ and $\gamma$, $\delta$ will be said to be in the span of a set of Jacobi fields $\mathcal{J}$ if it is possible to find a linear combination of Jacobi fields in $\mathcal{J}$, as defined in Subsection \ref{linearcombinations}, that determines a deformation of $\gamma$ into $\delta$ through one parameter families. 

\end{defn}
  
\begin{thm}\label{switcheroo}
Linear combinations of Jacobi fields and their restrictions span the space of geodesic segments connecting $m_1$ to $m_2$.
\end{thm}
\begin{proof}
Let $\gamma$ be the unique middle path in the family of geodesic segments connecting $m_1$ to $m_2$, and let $\delta$ be the geodesic segment $m_{1},\delta_{1}, \delta_{2},\ldots, m_{2}$. This theorem is proven by showing that there is a linear combination of Jacobi fields that determines $\gamma-\delta$. \\

It is clear that if $\delta$ is constructed by repeatedly surgering along arcs on the boundary of $S_{min}$ and discarding the null homologous multicurves $\partial S_{min}$, there is a linear combination of Jacobi fields coming from alternative surgeries that represent the difference of the two geodesic segments. Similarly, whenever for each $i$, $\delta_{i+1}$ could be constructed by surgering along a set of arcs whose endpoints are all assigned the same value of the overlap function, as in Example \ref{replace}; a surgery of this type is a surgery along the arcs on the boundary of $S_{imin}$ or $S_{imax}$ for some $i$. Also, the statement of the theorem is clear when it is possible to reduce to one of these previous cases by subtracting Jacobi fields representing optional surgeries or by adding/subtracting Jacobi fields that represent discarding null homologous submulticurves.\\

If an optional surgery $s_a$ on $\delta_{l}$ does not define a one parameter family, for the following special case it will be explained how to find a linear combination of Jacobi fields that take $\delta$ to a geodesic for which $s_a$ \textit{does} determine a one parameter family. Suppose, for some $l<k$, $\{a_{j}^{k}\}$ can be chosen such that 
\begin{itemize}
\item for each $j$ the endpoints of the arcs have the same value $f$ of the overlap function of $m_{2}-m_{1}$, and
\item $f$ is the value of the overlap function on the endpoints of $a$.
\end{itemize}
This special case occurs, for example, when all surgeries except $s_a$ are along arcs on the boundary of $S_{imax}$ or $S_{imin}$. There is a Jacobi field $J$ coming from an alternative surgery that replaces surgeries along the arcs $\{a_{j}^{k}\}$ with surgeries on $\delta_{m}$, $k\leq m$, along arcs $\{c_{j}^{m+1}\}$ with the same end points as $\{a_{j}^{k}\}$, but on the other side of $\delta_m$. Deforming in the direction of $J$, a geodesic segment is obtained along which $s_a$ determines a one parameter family. Subtract the corresponding Jacobi field to obtain a geodesic segment with one fewer optional surgeries than $\delta$.\\

Now if the previous special case does not occur, and $s_a$ is the only optional surgery along $\delta$, the remainder of this proof, applied to the geodesic segment connecting $\delta_{l+1}$ to $m_2$, shows how to reduce to the special case from the previous paragraph. If there is more than one optional surgery, let $s_a$ be the optional surgery performed on $\delta_{l}$, where $\delta_{l}$ is the last multicurve representing a vertex of $\delta$ along which optional surgeries are performed. Whenever two or more optional surgeries are performed on $\delta_l$, the corresponding arcs are necessarily on the same side of $\delta_l$, so this is not a problem. Next the second last optional surgery is removed, etc.\\

Warning - in the previous paragraph, what we may not do is restrict to subsegments connecting, for example, $\delta_{i}$ and $\delta_j$. The labels ``optional surgery'', ``alternative surgery'', etc. are not preserved when restricting to subsegments, because these labels refer to properties of the overlap function with $m_2$.\\

\begin{figure}
\begin{center}
\def\svgwidth{12cm}
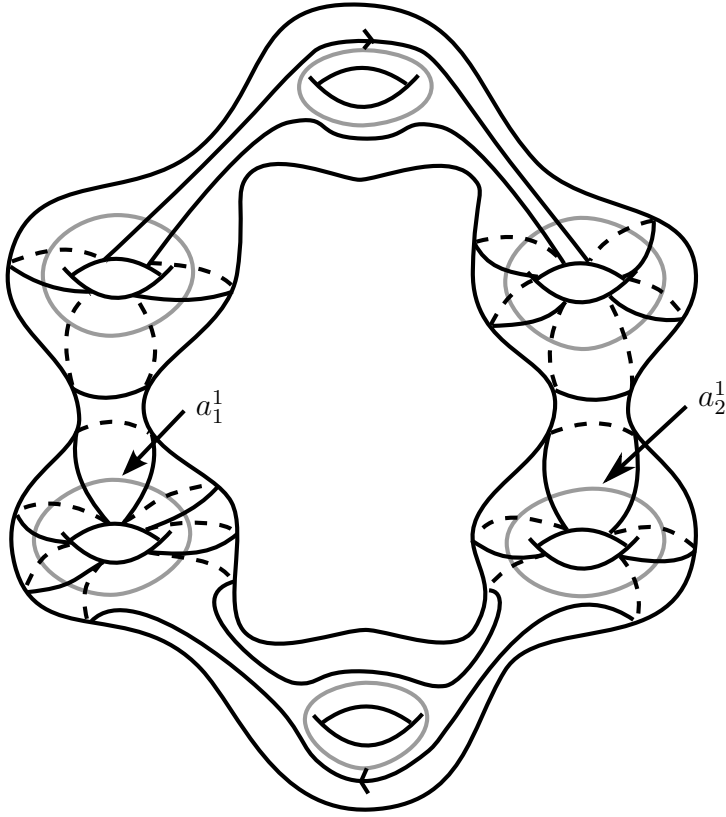
\caption{The multicurve $m_2$ is shown in grey.}
\label{quer}
\end{center}
\end{figure}

The main difficulty in proving this theorem comes from examples such as that in Figure \ref{quer}. Suppose $\delta_1$ is constructed by surgering $m_1$ along the arcs $\{a_{j}^{1}\}$ to the right of $m_1$. These arcs con not be chosen such that their endpoints have the same value of the overlap function, and none of the associated surgeries are optional. Call such sets of surgeries \textit{diagonal}. Deforming geodesic segments with diagonal surgeries into a middle path is difficult because these surgeries can not be replaced by surgeries along $S_{imax}$ or $S_{imin}$ for any $i$, nor by subtracting Jacobi fields coming from optional surgeries. \\

The endpoints of the arcs $\{a_{j}^{1}\}$ necessarily separate $S_{max}$ from $S_{min}$ on both multicurves $m_1$ and $m_2$. Otherwise surgering $m_1$ along the set $\{a_{j}^{1}\}$ could not give a vertex $\delta_1$ with $\delta(\delta_{1},m_{2})<\delta(m_{1}, m_{2})$. It follows that $\delta_1$ is a union of two disjoint multicurves; $\delta_{1+}$, which has arcs on the boundary of $S_{max}$, and $\delta_{1-}$, which has arcs on the boundary of $S_{min}$. Further diagonal surgeries along arcs with endpoints on $\delta_{1+}$ or $\delta_{1-}$ (but not both) to obtain the multicurve $\delta_2$ will also clearly preserve the decomposition, etc.\\


It is possible to keep performing diagonal surgeries to construct consecutive vertices along $\delta$ until a value of $i$, call it $i^{*}$, is reached such that no diagonal surgery on $\delta_{i^{*}}$ can be used to construct the next vertex along $\delta$. This happens when the maximum of the overlap function has been brought so low that $\delta_{i^{*}-}$ has an arc on the boundary of $S_{i^{*}max}$ and the minimum of the overlap function  brought so high that $\delta_{i^{*}+}$ has an arc on the boundary of $S_{i^{*}min}$.\\

Since the surgeries on different multicurves commute, by moving $\delta$ through one parameter families to the geodesic segment $\eta$, it is possible to assume without loss of generality that the surgeries on the ``$-$'' multicurves were performed before those on the ``$+$'' multicurves. Similarly, if diagonal surgeries on $\delta$ are interspersed with other surgeries, $\eta$ is chosen such that the diagonal surgeries were all performed first. Now since it is not possible to perform any more diagonal surgeries, it can be assumed with out loss of generality that for $i^{*}\leq i$, $\eta_{i+1}$ is constructed by surgering $\eta_{i}$ along the arcs on the boundary of $S_{imin}$.\\

Let $\eta_{k}$, $k \leq i^{*}$ be the first vertex of $\eta$ at which we start surgering along the ``$+$'' multicurves. By construction, there is an arcs of $m_{2}\cap(S\setminus \eta_{k+})$ on the boundary of $S_{kmin}$. The arcs of  $m_{2}\cap(S\setminus \delta_{k-})$ on the boundary of $S_{kmin}$ represent optional surgeries and determine a one parameter family over $\eta$. Move $\eta$ into this one parameter family to obtain $\eta^{1}$. On $\eta^{1}$, the surgeries along $\eta_{k+}$ along arcs not on the boundary of $S_{kmin}$ become optional surgeries that determine a one parameter family. Remove these optional surgeries by moving through the corresponding one parameter families to get a geodesic segment $\mu$, where $\mu_{k+1}$ was constructed from $\mu_{k}$ by surgering along arcs on the boundary of $S_{kmin}$. Similarly for $\mu_{k+2}$ up to $\mu_{i^{*}}$.\\

Now starting with $\mu_{k}$, replace the surgeries along arcs of $m_{2}\cap(S\setminus \mu_{k})$ on the boundary of $S_{kmin}$ with the surgeries along the arcs of $m_{2}\cap(S\setminus \mu_{k})$ on the boundary of $S_{kmax}$. Do the same with $\mu_{k+1}, \mu_{k+2}$ etc. until $\mu_{j}$ is reached, where $\mu_{j-}$ has an arc on the boundary of $S_{jmax}$. This is done by moving through one parameter families to get to the geodesic segment $\mu^{1}$. It can be assumed without loss of generality that a $\mu_{j}$ is reached before $m_2$, because otherwise this same argument, only with $+$ and $-$, min and max, and left and right interchanged, would apply. Along $\mu^{1}$, the surgeries used to construct the first $k$ multicurves commute with the surgeries used to construct the next $j-k$ multicurves, so again, moving through one parameter families, it is possible to exchange the order, to obtain a geodesic segment $\nu$. Along $\nu$, the same argument given before shows that it is possible to get rid of the diagonal surgeries by moving through one parameter families.\\

We have now covered all the different types of surgeries or ways of discarding null homologous multicurves that might be used to construct a geodesic path, and shown that there exist linear combinations of Jacobi fields that take vertices on all these geodesics to corresponding vertices on the middle path.
\end{proof}

\begin{defn}[Dimension of the space of geodesic segments]\label{dimension}
The dimension of the space of geodesic segments in $\mathcal{C}(S,\alpha)$ connecting the vertices $m_1$ to $m_2$ is the largest possible dimension of the space of Jacobi fields along a geodesic segment.
\end{defn}

\begin{defn}[Critical Level]
The index $i$ is a critical level if $\gamma_{i}$ is the first or last vertex in the support of a Jacobi field $J(a,\gamma)$, where $J(a,\gamma)$ is not the restriction of another Jacobi field.
\end{defn}


The index $i$ could be a critical level if, for example, the vertex after $\gamma_{i-1}$ could not have been constructed by surgering along a set of arcs of the form $\{v_{j}\circ b_{j}^{i-1}\circ w_{j}\}$, where $b_{j}^{i-1}$ is homotopic to $a_{j}^{i-1}$, or when the number of arcs in the homotopy class with representative $v_{j}\circ a_{j}^{i-1}\circ w_{j}$ is not the same as the number of arcs in the homotopy class with representative $a_{j}^{i-1}$ for some $j$.\\


\textbf{Remark}. There are two possible ways in which the dimension of the space of geodesic segments could have been defined. Firstly, in terms of the maximum possible number of Jacobi fields along a geodesic segment as in Definition \ref{dimension}, and secondly, in terms of the maximum number of Jacobi fields needed in a linear combination representing the difference of two geodesic segments. Analysing the proof of Theorem \ref{switcheroo} carefully shows that, assuming Theorem \ref{maintheorem}, both are finite. This is because it is possible to move $\delta$ through a finite number of one parameter families to a geodesic segment $\omega$ for which the following is true: For all $i$, $\omega_{i+1}$ is constructed from $\omega_{i}$ by the obvious continuation of the construction of $\omega_{i}$ from $\omega_{i-1}$ unless $\omega_{i}$ is a critical level. Then the deformations that take a vertex $\omega_{i+1}$ to its target vertex $\gamma_{i+1}$ are the obvious continuations (i.e. deformations within the same one parameter family) of the deformations needed to take $\omega_{i}$ to its target vertex $\gamma_{i}$, unless a critical level is reached.\\

It follows from the remark that when geodesic segments with the same endpoints do not stay close, there will necessarily be some Jacobi field with large magnitude.\\

We now begin the proof of Theorem \ref{maintheorem}.\\

\begin{proof}
It is well known that the number of homotopy classes of arcs of $m_{2}\cap (S\setminus m_{1})$ is bounded. For example, in \cite{Me2}, Lemma 11, the sharp bound $-3\chi(S)$ was obtained. To see how the number of homotopy classes of horizontal arcs bounds the dimension of the space of Jacobi fields, first of all, surgeries along homotopic arcs are equivalent. If $\gamma_i$ is surgered along a set of horizontal arcs $\{a_{j}^{i+1}\}$ containing the arc $a$, the multicurve $\gamma_{i+1}$ is not obtained by also surgering along $v\circ a \circ w$ for arcs $v$, $w$ of $m_{2}\cap (S\setminus \gamma_{i})$ because
\begin{itemize}
\item if $a$ has both endpoints on a curve $c$ in $\gamma_{i}$ such that $\gamma$ has more than one curve homotopic to $c$, then $\gamma_{i+1} - \gamma_{i}$ could not be the boundary of an embedded, oriented subsurface of $S$.
\item if $a$ has both endpoints on the null homologous curve $N(\{a_{j}^{i+1}\})$ discarded after surgering along $\{a_{j}^{i+1}\}$, since $N(\{a_{j}^{i+1}\})$ is discarded anyway, it does not make any difference to the path if we surger it along $v\circ a \circ w$ or not.
\item in all other cases, surgering along $v\circ a \circ w$ would mean that $\gamma_{i+1}$ intersects $\gamma_{i}$.
\end{itemize}


Local extrema of the overlap function can not ever be created as $i$ increases; surgering along a horizontal arc of $m_{2}\cap(S\setminus \gamma_{i})$ to the right of $\gamma_{i}$ decreases a local maximum along $m_2$, and surgering $\gamma_{i}$ along a horizontal arc to the left of $\gamma_{i}$ increases a local minimum along $m_2$. A saddle is a local extremum along $m_2$, so for the same reason, the number of saddles can not increase either. However, not all saddles or local extrema determine independent surgeries, because many of them might have homotopic arcs on their boundaries. In Figure \ref{octagon} is an example of how the number of \textit{homotopy classes} of horizontal arcs can increase. \\

\begin{figure}
\begin{center}
\def\svgwidth{12cm}
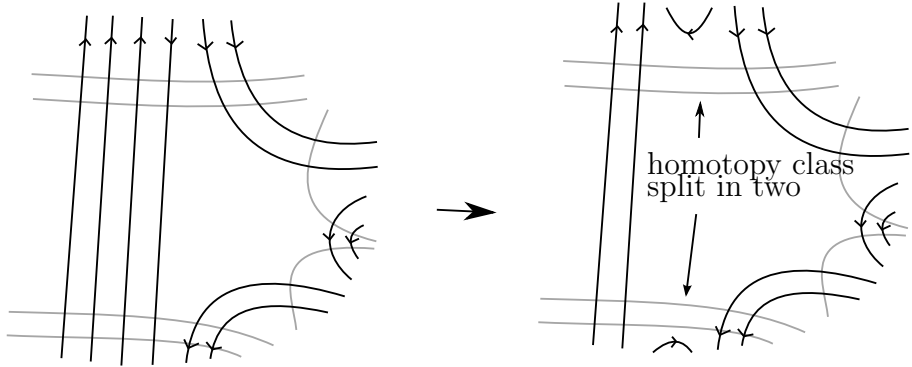
\caption{The multicurve $m_2$ is shown in grey, and $m_1$ in black. After surgering along a horizontal arc of $m_{2}\cap (S\setminus m_{1})$, the number of homotopy classes of horizontal arcs increases.}
\label{octagon}
\end{center}
\end{figure}

\textbf{Splitting and Killing homotopy classes}. For a given arc $a$ in the set $\{a_{j}^{1}\}$, suppose $v_{1}\circ a \circ w_{1}$ is an arc in the set $\{a_{j}^{2}\}$, and $v_{2}\circ v_{1}\circ a \circ w_{1}\circ w_{2}$ an arc in $\{a_{j}^{3}\}$, etc. For large enough $n$, one or both of the following two things will happen: there are two or more homotopy classes of arcs $v^{'}_{n}\circ\ldots\circ a \circ w_{1}\circ\ldots\circ w_{n}$ and $v^{"}_{n}\circ\ldots\circ a \circ w_{1}\circ\ldots\circ w_{n}$ or $v_{n}\circ\ldots\circ a \circ w_{1}\circ\ldots\circ w^{'}_{n}$ and $v_{n}\circ\ldots\circ a \circ w_{1}\circ\ldots\circ w_{n}^{"}$; this will be called \textit{splitting} the homotopy class $a$. The other possibility is that $v_{n}\circ\ldots\circ a \circ w_{1}\circ\ldots\circ w_{n}$ is a vertical arc, but $v_{n-1}\circ\ldots\circ a \circ w_{1}\circ\ldots\circ w_{n-1}$ was not. This will be called \textit{killing} the homotopy class $v_{n-1}\circ\ldots\circ a \circ w_{1}\circ\ldots\circ w_{n-1}$. A homotopy class is killed when one, but not both, of $v_n$ or $w_n$ is a horizontal arc.\\

Given that the number of homotopy classes of arcs is bounded from above by $-3\chi(S)$, the number of critical levels that arise from splitting a given homotopy class is clearly bounded by $-3\chi(S)$. So $a$ is split into fewer than $-3\chi(S)$ homotopy classes, many of which will eventually be killed. Once a homotopy class of horizontal arcs has been killed, the resulting homotopy classes of vertical arcs can become $v_{i}$s and $w_{i}$s for another horizontal arc, and the hexagons, octagons etc. that split $a$ into homotopy classes, can cause another homotopy class of arcs to be split.\\

The geometrical significance of killing off a homotopy class of arcs is that all the local maxima or minima of the overlap function on $m_2$ corresponding to that homotopy class have been levelled off. Let $k$ be the number of homotopy classes a homotopy class of horizontal arcs of $m_{2}\cap(S\setminus m_{1})$ with representative $a$ is eventually split up into before the arcs are all killed off. Since surgering the multicurve $\gamma_{i}$ along horizontal arcs of $m_2$ can not create local extrema of the overlap function along $m_2$, and up to homotopy, there were no more than $-3\chi(S)$ arcs of $m_{2}\cap (S\setminus m_{1})$ representing local extrema on $m_2$, it follows that the $k$ splits do not need to be counted more than $-3\chi(S)$ times.\\

This gives a bound of $9\chi(S)^{2}$ for the number of Jacobi fields along $\gamma$ coming from optional surgeries. Some of these optional surgeries, when grouped together, might determine Jacobi fields coming from alternative surgeries. Also, a given homotopy class of horizontal arcs might determine a surgery that is performed as a component of more than one alternative surgery.\\

There can be no more than $-\chi(S)-1$ Jacobi fields from alternative surgeries with support on $\gamma_1$. This comes from the observation used in the proof of Theorem \ref{switcheroo}, that an alternative surgery determines a null homologous multicurve ($\partial S_{+}$) that partitions $\gamma_1$ into two multicurves. For surgeries along arcs on the boundary of $S_{max}$ or $S_{min}$, $\partial S_{+}$ could be contractible, giving a trivial partition. Jacobi fields coming from alternative surgeries could arise from a splitting or killing of a homotopy class, or when a local extremum makes it necessary to change the number of arcs to be surgered along. This gives an upper bound of $18\chi(S)^{2}$ Jacobi fields coming from alternative surgeries. An upper bound on the number of isotopy classes of null homologous submulticurves giving linearly independent Jacobi fields is half the number of Jacobi fields coming from alternative surgeries. In total, this gives a bound of $36\chi(S)^{2}$.\\


\end{proof}

\textbf{Remark}. The bound in the previous proof is clearly not sharp. However, to get a considerably better bound, it would seem that a much more detailed argument would be needed; the details of which are more tedious than illuminating.\\

\section{Sublevel Projections}\label{sublevel}

Subsurface projections were defined in \cite{MasurandMinskyII} in order to be able to break the curve complex down into simpler pieces, thought of as curve or arc complexes of subsurfaces. The nested structure arising from the subsurface projections were used to describe families of quasigeodesics called hierarchy paths, and to show how these families of quasigeodesics are controlled by the subsurface projections of their endpoints.\\

In this section, the notion of sublevel projections are defined, so-named because there are some very strong parallels with subsurface projections. Informally, critical levels are used to partion a geodesic into subintervals that are as rigid as possible and behave almost independently of each other.\\


Let $m_{1}, \gamma_{1}, \ldots, m_{2}$ be the middle path connecting $m_1$ and $m_2$. Given two integers $l_{1} <l_2$ in the range of the overlap function of $m_{2}-m_1$, the \textit{sublevel projection} of $m_{1}$ and $m_{2}$ between the levels $l_1$ and $l_2$, $\Pi_{l_{1}}^{l_{2}}(m_{1},m_{2})$, is the pair of homologous multicurves $(\gamma_{l_{1+1}}, \gamma_{l_{2}})$.\\

The sublevel projection of $m_1$ and $m_2$ between the levels $l_1$ and $l_2$ is similar to a subsurface projection to $S_{l_{1+1}\leq f \leq l_{2}}$, in the sense that $\gamma_{l_{1}}$ and $\gamma_{l_{2}}$ represent vertices as close as possible to $m_1$ and $m_2$, respectively, given that they only intersect within the subsurface $S_{l_{1+1}\leq f \leq l_{2}}$. It follows from Theorem 9 in \cite{Me2} that this definition is symmetric in $m_1$ and $m_2$.\\

\textbf{Distance Formula}. Consider the finite number of sublevel projections of the form $\Pi_{i}:=\Pi_{l_{i}}^{l_{i+1}}(m_{1},m_{2})$, where $l_{i}$ and $l_{i+1}$ are critical levels. Any collection of surgeries performed on the multicurve $\gamma_i$ to construct a multicurve $\gamma_{i+1}$ with $d(\gamma_{i+1},\gamma_{n})=d(\gamma_{i},\gamma_{n})-1$ necessarily decreases the distance between $\gamma_i$ and $\gamma_n$ in one of the sublevel projections $\Pi_i$. A distance formula analogous to the distance formula from \cite{MasurandMinskyII}, with a uniform bound on the number of sublevel projections follows immediately from the construction and Corollary \ref{j}. In this way, families of tight paths in $\mathcal{C}(S,\alpha)$ are even more rigidly controlled by the sublevel projections of their endpoints than is the case in the marking graph for hierarchy paths under subsurface projections, \cite{MasurandMinskyII}.

\bibliographystyle{plain}

\bibliography{bibgeospace}
\end{document}